\documentclass[letterpaper, 10 pt, conference]{ieeeconf}
\usepackage{ifpdf}

\usepackage{algorithm}
\usepackage{algpseudocode}
\usepackage{color}
\usepackage{url}
\ifpdf
\usepackage[pdftex]{graphicx}
\else
\usepackage[dvips]{graphicx}
\fi
\usepackage{tabularx}
\DeclareGraphicsExtensions{.eps}
\usepackage[caption=false,font=footnotesize]{subfig}
\usepackage{comment}
\usepackage{bm}
\usepackage{amsmath,amsthm,amsfonts,amssymb}
\usepackage[noadjust]{cite}
\usepackage{nomencl}
\theoremstyle{definition}
\newtheorem{theorem}{\textbf{Theorem}}

\newtheorem{proposition}[theorem]{\textbf{Proposition}}

\newtheorem{corollary}[theorem]{\textbf{Corollary}}

\usepackage{tikz}
\usetikzlibrary{automata,positioning}

\newcommand{\up}[1]{^\mathrm{#1}}
\newcommand\Tstrut{\rule{0pt}{2.6ex}}         
\newcommand\Bstrut{\rule[-0.9ex]{0pt}{0pt}}   

\IEEEoverridecommandlockouts
\begin{document}

\title{A Lagrangian Policy for Optimal Energy Storage Control}

\author{Bolun~Xu,
	Magnus Korp\aa s,
	Audun Botterud,
	Francis O'Sullivan
\thanks{B.~Xu, A.~Botterud, and F.~O'Sullivan are with Massachusetts Institute of Technology, MA, USA. M.~Korp\aa s is with Norwegian University of Science and Technology,  Trondheim, Norway. Contact: \{xubolun, audunb,frankie\}@mit.edu, magnus.korpas@ntnu.no. }
}

\maketitle
\makenomenclature

\begin{abstract}
This paper presents a millisecond-level look-ahead control algorithm for energy storage. The algorithm connects the optimal control with the Lagrangian multiplier associated with the state-of-charge constraint. It is compared to solving look-ahead control using a state-of-the-art convex optimization solver. The paper include discussions on sufficient conditions for including the non-convex simultaneous charging and discharging constraint, and provide upper and lower bounds for the primal and dual results under such conditions. Simulation results show that both methods obtain the same control result, while the proposed algorithm runs up to 100,000 times faster and solves most problems within one millisecond. The theoretical results from developing this algorithm also provide key insights into designing optimal energy storage control schemes at the centralized system level as well as under distributed settings.


\end{abstract}


\IEEEpeerreviewmaketitle

\section{Introduction}
Energy storage devices such as batteries are key resources in  future energy systems due to their flexibility and fast response speed, and their convenient installations as either large-scale bulk units or as distributed resources.
Most real-time energy storage operations are optimized using predictive control, with applications such as economic dispatch~\cite{korpas2006operation}, frequency control~\cite{li2016connecting}, voltage control~\cite{zhang2015optimal}, renewable integration~\cite{khalid2010model}, energy arbitrage~\cite{krishnamurthy2018energy}, peak shaving~\cite{shi2018using}, electric vehicle charging~\cite{xu2016hierarchical}, or a combination of several aforementioned applications~\cite{megel2014scheduling}. These predictive control strategies solve a multi-period optimization problem over a look-ahead horizon at each control step, obtaining the control and state profile over the entire horizon but only applies the first control result, the problem is then updated with a new horizon and state information for the next control step.


The challenge of using look-ahead control in practice is trading off optimality with computational tractability, as a longer look-ahead horizon incorporates more future information and thus improves solution optimality, but increases the computational challenge significantly. For example, real-time economic dispatches in power systems are typically solved over a single period or with a look-ahead horizon less than one hour~\cite{epri_2016}. However, power system operations have strong daily patterns due to load and weather variations, such as charging storage from solar power during the day and discharge during the night. Thus, being able to incorporate a look-ahead horizon over one day or even longer is crucial for the future power system, but solving such problems over the scale of a realistic power system is extremely computationally challenging~\cite{zhao2018multi}, especially binary variables must be introduced in certain application to prevent simultaneous charging and discharging, making the problem non-convex~\cite{castillo2013profit}. In addition, future uncertainties in power systems are often modeled with scenarios~\cite{wang2011wind}, and modeling uncertainties from multiple sources can easily lead to hundreds of scenarios that makes it almost impossible to solve look-ahead economic dispatch with conventional optimization solvers. While methods such as stochastic dual dynamic programming~\cite{papavasiliou2018application} reduce the solution complexity by introducing inter-temporal and scenario decomposition, the computation is still difficult and requires significant memory usage. On the other hand, the optimal control problem must be solved within a reasonable timescale to fully utilize the fast response speed of energy storage devices. For example, a battery ramps from zero to full discharge power within milliseconds~\cite{gao2002dynamic} thus, a scheme that takes seconds or even minutes to update the control decision is not appropriate for controlling batteries.

Solving storage control from the dual problem is more effective than dealing with the primal problem directly since the storage has only a single state variable with upper and lower bounds. Cruise et a
l.~\cite{cruise2014optimal} concluded the storage control problem can be solved using a search algorithm based on the binding conditions on the state-of-charge, and Hashimi et al~\cite{hashmi2017optimal} has developed an algorithm for energy storage price arbitrage with quadratic time complexity, based on solving the dual problem.
Comparably, the technical contributions of this paper and the main advantages of the proposed algorithm is summarized as follows:
\begin{enumerate}
    \item We show that energy storage control with a generalized time-varying objective functions can be solved in worst-case linear time complexity and constant space complexity, with respect to the look-ahead horizon.
    \item We conclude the optimal control condition for energy storage without having to go through the full look-ahead horizon, i.e., the current control is optimal with respect to any future realizations that may not be included in the current look-ahead window.
    \item We derive a sufficient condition for the occurrence of simultaneous charging and discharging, and provide upper and lower bounds for the prime and dual results under such non-convex conditions.
\end{enumerate}


The rest of this paper is organized as follows: Section~II formulates the problem; Section~III presents main analytical results and the algorithm; Section~IV demonstrates numerical results; and Section~V concludes the paper.

\section{Formulation and Preliminaries}

\subsection{Problem Formulation}
We consider a time period $t\in \mathcal{T}=\{1,\dotsc, T\}$ where $t=1$ is the current control step and $t=2$ to $T$ is the look-ahead horizon.  The optimal control profile $p^*_t$ is a minimizer to the following multi-period optimization problem
\begin{subequations}\label{mso}
\begin{gather}
    p^*_t \in \arg \min_{p_t}\; \sum_{t=1}^T O_t(p_t) + C_T(e_T) \label{mso:obj}\\
    \text{s.t.}\nonumber\\
    p_t = p^+_t - p^-_t \label{mso:c1}\\
    p^+_t = 0 \text{ or } p^-_t = 0 \;\forall t\in \mathcal{T} \label{mso:c0}\\
    e_t - e_{t-1} =   -p_t^+/\eta + p_t^-\eta: \theta_{t-1} \label{mso:c2}
    \end{gather}
    \begin{equation}\label{mso:c3}
    \begin{aligned}
    p^+_t &\geq 0 : \underline{\mu}^+_t\,, \\
    p^-_t &\geq 0 : \underline{\mu}^-_t\,, \\
    e_t &\geq 0 : \underline{\nu}_t\,,
    \end{aligned}
    \quad
    \begin{aligned}
     p^+_t &\leq P : \overline{\mu}^+_t  \\
     p^-_t &\leq P : \overline{\mu}^-_t \\
     e_t &\leq E : \overline{\nu}_t
    \end{aligned}
    \end{equation}
\end{subequations}
where
\begin{enumerate}
    \item $O_t(\cdot)|\mathbb{R}\to\mathbb{R}$ is a scalar time-varying convex objective function. Its derivative is denoted as $o_t = \dot{O}_t$.
    \item $C_T(\cdot)|\mathbb{R}\to\mathbb{R}$ is the terminal cost function of the end state of charge $e_T$. $C_T$ is also convex and its derivative is denoted as $c_T = \dot{C}_T$.  Note that $C_T$ can also be used to model the operation beyond $T$ via dynamic programming~\cite{bertsekas2005dynamic}.
    \item $p_t$ is the control decision variable and it is the energy dispatched from the storage during the time period $t$.
    \item $p^+_t$ is the positive (discharge) component of $p_t$.
    \item $p^-_t$ is the negative (charge) component of $p_t$.
    \item $e_t$ is the state of charge (SoC) at the end of time period $t$, subjects to an initial value of $e_0$.
    \item $\eta \in (0,1]$ is the storage charge and discharge efficiency.
    \item $P \in \mathbb{R}^+$ is the maximum energy that can be charged or discharged into the storage during a single period.
    \item $E \in \mathbb{R}^+$ is the maximum energy that can be stored in the storage.
    \item $p^*_t$ is a set of minimizers to the optimization problem.
    \item $\theta_{t}$ is the Lagrangian multiplier associated with the SoC dynamic, its physical meaning is the marginal value of SoC at the end of time $t$ over the future operation $[t+1, T]$.
    \item $\underline{\mu}^+_t$, $\overline{\mu}^+_t$, $\underline{\mu}^-_t$, $\overline{\mu}^-_t$, $\underline{\nu}_t$, $\overline{\nu}_t$ are positive dual variables associated with inequality constraints. 
\end{enumerate}
The objective function \eqref{mso:obj} minimizes the total operating cost over the period $[1, T]$. Constraint \eqref{mso:c1} divides the control $p_t$ into a positive component $p^+_t$ and a negative component $p^-_t$ in order to model the efficiency difference during charge and discharge in the SoC evolution constraint. \eqref{mso:c0} is the non-simultaneous charging and discharging constraint that enforces the storage to only charge or discharge at any given time point. \eqref{mso:c2} models the SoC evolution subjects to efficiencies. Power and energy ratings are modeled in \eqref{mso:c3}.

\begin{figure*}[!t]
    \centering
    \subfloat[]{
        \includegraphics[trim={5mm 0 10mm 3mm},clip, width = .95\columnwidth]{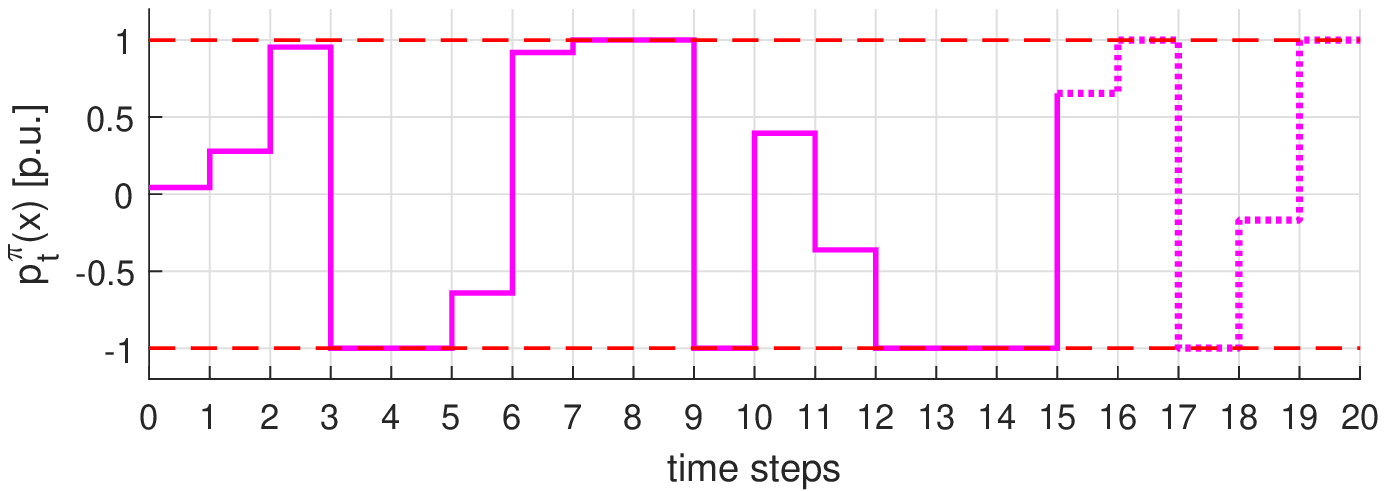}
        \label{fig:th1a}
    }
    \subfloat[]{
        \includegraphics[trim={5mm 0 10mm 3mm},clip, width = .95\columnwidth]{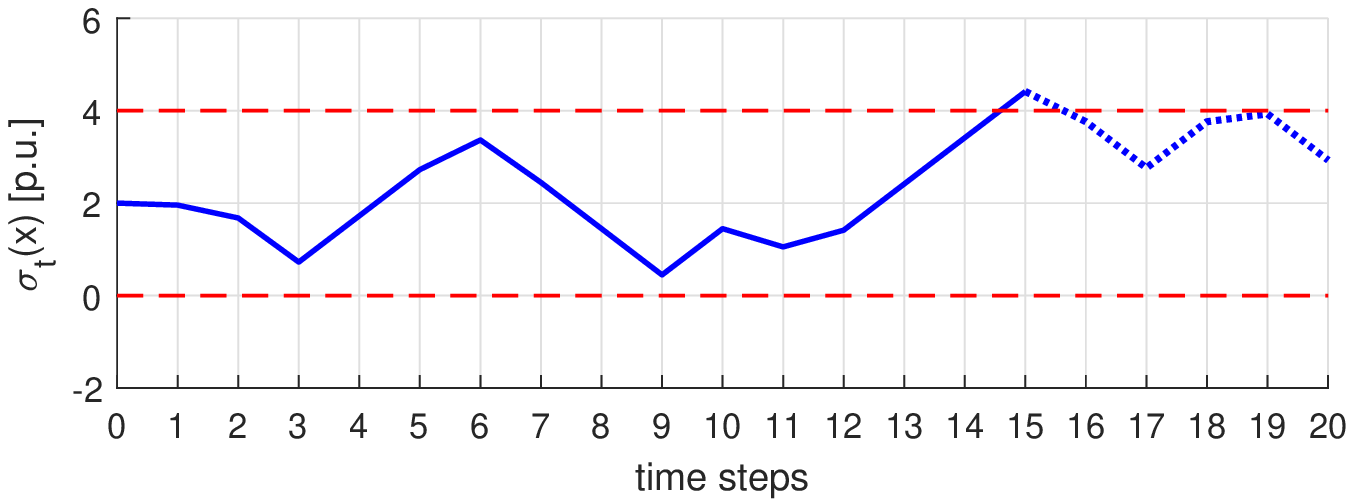}
        \label{fig:th1b}
    }
    \\
    \subfloat[]{
        \includegraphics[trim={5mm 0 10mm 3mm},clip, width = .95\columnwidth]{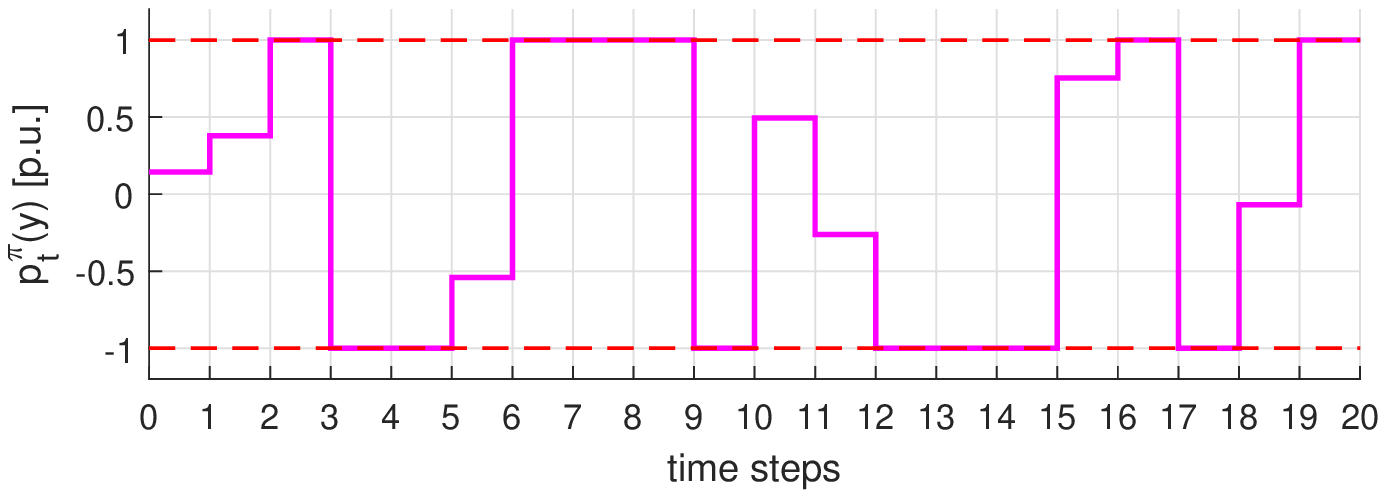}
        \label{fig:th1c}
    }
    \subfloat[]{
        \includegraphics[trim={5mm 0 10mm 3mm},clip, width = .95\columnwidth]{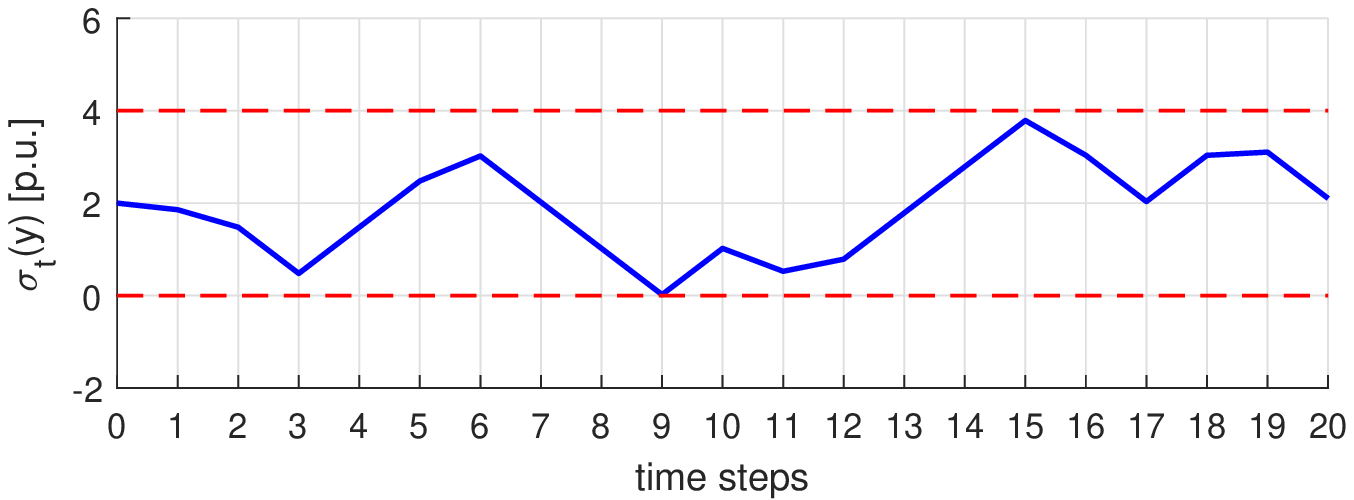}
        \label{fig:th1d}
    }
    \\
    \subfloat[]{
        \includegraphics[trim={5mm 0 10mm 3mm},clip, width = .95\columnwidth]{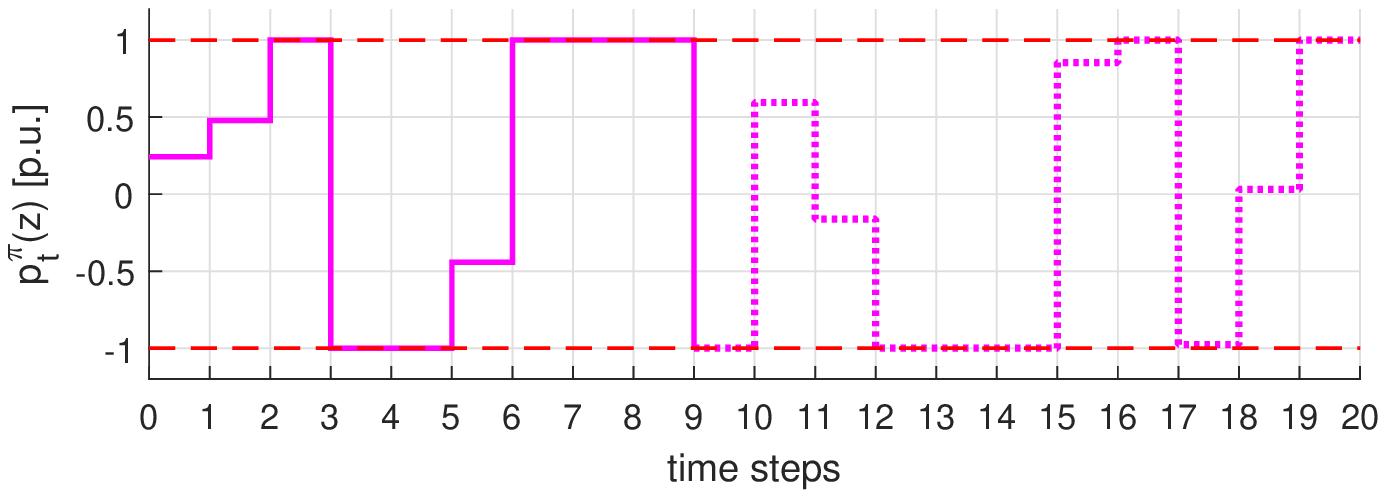}
        \label{fig:th1e}
    }
    \subfloat[]{
        \includegraphics[trim={5mm 0 10mm 3mm},clip, width = .95\columnwidth]{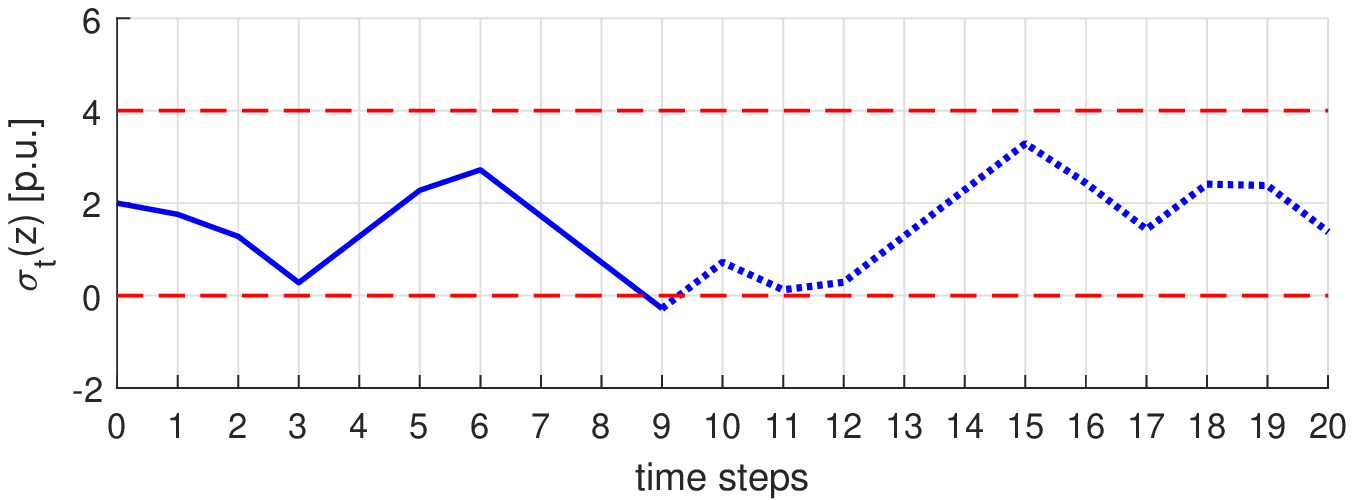}
        \label{fig:th1f}
    }
    \caption{\footnotesize Illustration of conditions listed in Theorem~\ref{th1} with three guesses of the Lagrangian value: $x$, $y$, $z\in\mathbb{R}$, $(x>y>z)$ to the policy $\pi$ and the control simulation results are plotted respectively (a): $p\up{\pi}_t(x)$; (b): $\sigma_t(x)$; (c): $p\up{\pi}_t(y)$; (d): $\sigma_t(y)$; (e): $p\up{\pi}_t(z)$; (f): $\sigma_t(z)$. All trials use the same $O_t$ and storage setting, the power limit is between -1 to 1 and the SoC limit is between 0 to 4, which are plotted with red dashes in the figures. As shown in (a) and (b), $\sigma_t(x)$ reached the upper SoC bound at $t=15$, indicating $x\geq\theta_0$ and the rest of the simulation is plotted in dots indicating it is not required. In (c) and (d),  $\sigma_t(y)$ reached the end $t=20$ without hitting either bound, thus the equality relationship between $y$ and $\theta_0$ can be concluded using $c_T(\sigma_T(y))$. In (e) and (f), $\sigma_t(z)$ hit the lower SoC limit at $t=9$ indicating $z\leq \theta_0$, and the remaining simulation is again plotted with dots.}
    \label{fig:th1}
\end{figure*}

\subsection{Karush-Kuhn-Tucker conditions}

The results in this paper are primarily based on the use of the Karush-Kuhn-Tucker (KKT) conditions~\cite{boyd2004convex}, which are listed below for \eqref{mso} ( the non-simultaneous charging and discharging constraint \eqref{mso:c0} is non-convex and is excluded from the KKT condition below, treatment of this constraint will be discussed later):
\begin{subequations}\label{mfv}
\begin{align}
    o_t(p^*_t) + \overline{\mu}^+_t - \underline{\mu}^+_t + \theta_{t-1}/\eta &= 0 : p^+_t \label{mfv_eq2}\\
    -o_t(p^*_t) + \overline{\mu}^-_t - \underline{\mu}^-_t - \theta_{t-1}\eta &= 0 : p^-_t\label{mfv_eq3}\\
    \theta_{t-1} + \overline{\nu}_t - \underline{\nu}_t - \theta_{t} &= 0 : e_t \label{mfv_eq1}\\
    \theta_{T} + c_T(e^*_T) &= 0 : e_T \label{mfv_eq4}
\end{align}
\end{subequations}
and the complimentary slackness conditions associated with the inequality dual variables:
\begin{equation}\label{csc}
\begin{aligned}
\underline{\mu}^+_t [p^*_t]^+ &= 0 \,, \\
\underline{\mu}^-_t [-p^*_t]^+ &= 0 \,, \\
\underline{\nu}_t e^*_t &= 0\,,
\end{aligned}
\quad
\begin{aligned}
\overline{\mu}^+_t (P-[p^*_t]^+) &= 0 \,, \\
\overline{\mu}^-_t (P-[-p^*_t]^+) &= 0 \,, \\
\overline{\nu}_t (E-e^*_t) &= 0\,,
\end{aligned}
\end{equation}
and $\underline{\mu}^+_t$, $\overline{\mu}^+_t$, $\underline{\mu}^-_t$, $\overline{\mu}^-_t$, $\underline{\nu}_t$, $\overline{\nu}_t \geq 0$.
Note that we replaced the use of $p^+_t$ and $p^-_t$ with $p_t$ since $\partial p^+_t = \partial p_t$, $\partial p^-_t = -\partial p_t$, $p^+_t = [p_t]^+$ and $p^-_t = [-p_t]^+$, where $[x]^+ = \max\{0,x\}$ is the positive value function.

\section{Main Results}


We start by relaxing constraint~\eqref{mso:c0} so that the rest of the problem is convex. Thus we establish a closed-form connection between the primal and dual problem in Proposition~\ref{pro:pi}. We then present Theorem~\ref{th1} on identifying the equality relationship between $\theta_0$ and any real number $x\in\mathbb{R}$ using numerical simulation, and develop a binary search algorithm that finds the dual result $\theta_0$ and , thus, the primal result $p^*_1$. Then we discuss  how we can bound the result when it is necessary to incorporate the non-convex non-simultaneous charging and discharging constraint~\eqref{mso:c0} using the proposed algorithm.


\subsection{Optimal Control Policy}
We define the policy $\pi$ that calculates a storage control decision $p\up{\pi}_t(x)$ for time $t$ from an input $x\in\mathbb{R}$ as
\begin{gather}
    p^+_t(x) = \Big[\varphi_t(-x/\eta)\Big]^P_0,\; p^-_t(x) = \Big[-\varphi_t(-x\eta)\Big]^P_0 \\
    p^{\pi}_t(x) = p^+_t(x) - p^-_t(x)\label{es:policy}
\end{gather}
where $[x]^y_z = \max\{\min\{x,y\},z\}$ saturates $x$ between $y$ and $z$ ($z\leq y$),
and $\varphi_t(x) : \mathbb{R}\to \mathbb{R}$ is the inverse function of $o_t$ (derivative of $O_t$) as
\begin{align}
    \varphi_t(x) &= \sup\{y\in\mathbb{R} | o_t(y) \leq x\}\,.
\end{align}
Note that $\varphi_t$ is an alternative definition of the inverse function to $o_t$ while compatible with a piecewise linear $O_t$. 

The following proposition states that we can obtain the optimal control $p^*_t$ by using the Lagrangian multiplier $\theta_{t-1}$ as input to policy $\pi$:
\begin{proposition}\label{pro:pi}
     Policy $\pi$ is a minimizer to problem \eqref{mso} when using the Lagrangian $\theta_{t-1}$ as the input, i.e.,  $p^*_t = p^{\pi}_t(\theta_{t-1})$\,.
\end{proposition}
\begin{proof}
We start by rewriting the KKT condition associated with $p^+_t$ as
\begin{subequations}
\begin{align}
o_t(p^*_t)+\theta_{t-1}/\eta &< 0 \text{ only if } p^*_t = P \\
o_t(p^*_t)+\theta_{t-1}/\eta &> 0 \text{ only if } p^*_t = 0 \\
o_t(p^*_t)+\theta_{t-1}/\eta &= 0 \text{ if } 0\leq \varphi_t(-\theta_{t-1}/\eta)\leq P
\end{align}
\end{subequations}
where we substitute the complementary slackness condition into \eqref{mfv_eq2} that replaces $\overline{\mu}^+_t$ and $\underline{\mu}^+_t$. It is now trivial to see that we can calculate $p^*_t$ as $\varphi_t(-\theta_{t-1}/\eta)$ and limiting the result between 0 and $P$, hence
\begin{align}
p^+_t = \max\{0, \min\{P, \varphi_t(-\theta_{t-1}/\eta)\}\}\,.
\end{align}
We repeat the similar process for $p^-_t$ with \eqref{mfv_eq3}, and use $p_t = p^+_t-p^-_t$ which gives us the result in Proposition~\ref{pro:pi}.
\end{proof}

The following corollary supplements that with $\theta_{0}$ we can obtain $p^*_1$ as well as a series of consecutive optimal control decisions by recording the accumulated sum of the control results $\sigma_t$ defined as
\begin{align}\label{eq:th_soc}
    \sigma_t(x) =& \sigma_{t-1}(x) - [p^{\pi}_t(x)]^+/\eta - [p^{\pi}_t(x)]^-\eta\,,
\end{align}
with the initial value $\sigma_0(x) = e_0$, where $[x]^+ = \max\{0,x\}$ is the positive value function, and $[x]^- = \min\{0,x\}$ is the negative value function. $\sigma_t(x)$ emulates the SoC evolution but using the control result $p\up{\pi}_t(x)$ which may not be optimal. Another difference is that $\sigma_t(x)$ is not limited between $[0,E]$, instead, whether any $\sigma_t(x)$ falls above $E$ or below $0$ is an indicator on the optimality of $p\up{\pi}_t(x)$, as defined by the following corollary:
\begin{corollary}\label{cor:pi}
$p^*_t=p\up{\pi}_t(\theta_0)$ if $0 \leq \sigma_{\tau}(\theta_0) \leq E$ $\forall$ $\tau\in[1,t]$.
\end{corollary}
Corollary~\ref{cor:pi} means that we can maintain optimal control by using $\theta_0$ as the input to \eqref{es:policy} for control steps beyond $t=1$ if all previous $\sigma_t(\theta_0)$ are within the SoC constraint. Corollary~\ref{cor:pi} is based on Proposition~\ref{pro:pi} and the KKT condition associated with $e_t$ in \eqref{mfv_eq1} that the $\theta_t$ value will not change if both $\overline{\nu}_t$ and $\underline{\nu}_t$ are zeros, indicating $0\leq e^*_t \leq E$ and $\sigma_t(x) = e^*_t$. This corollary is thus proved.

\subsection{Main Theorem on Finding Lagrangian Dual}

\begin{theorem}\label{th1}
Given $x\in\mathbb{R}$, its equality relationship with respect to $\theta_0$ can be determined as
\begin{enumerate}
    \item  \textbf{If} a) $\sigma_t(x)$ reaches upper bound first, i.e.,  $\exists \tau\in\mathcal{T} \text{ s.t. } \sigma_{\tau}(x) > E \text{ and }
            0\leq\sigma_{\gamma}(x) \leq E$ $\forall \gamma \in [1,\tau)$;
     \textbf{or} b) $\sigma_t(x)$ reached neither bound and $x > -c_T(\sigma_T(x))$;
    \textbf{then} $x\geq \theta_0$;
    \item  \textbf{If} a) $\sigma_t(x)$ reaches lower bound first, i.e.,  $\exists \tau\in\mathcal{T} \text{ s.t. } \sigma_{\tau}(x) < 0 \text{ and }
        0\leq\sigma_{\gamma}(x) \leq E$ $\forall \gamma \in [1,\tau)$;
        \textbf{or} b) $\sigma_t(x)$ reached neither bound and $x < -c_T(\sigma_T(x))$;
    \textbf{then} $x\leq \theta_0$;
    \item \textbf{If} $\sigma_t(x)$ reached neither bound and $x = -c_T(\sigma_T(x))$, \textbf{then} $x = \theta_0$.
\end{enumerate}
\end{theorem}
Proof of this theorem is deferred to Appendix. The intuition is that the Lagrangian dual is the price of the stored energy, its value does not change despite the change with the SoC evolution, except reaching either the upper or lower SoC bound. The SoC series driven by the optimal dual value should never exceed the SoC bounds as the dual value itself reflects the constrained storage capacity. If the SoC exceed the upper SoC bound, it means SoC value is over estimated as the storage does not have enough capacity to store the excessive energy, hence we picked an $x$ that is higher than the optimal Lagrangian dual value. Vice versa, if the SoC exceed the lower bound, meaning we under estimated the dual value.

\subsection{Solution Algorithm}\label{app:alg}
We design a binary search algorithm that finds $\theta_0$ according to Theorem~\ref{th1}, thus  we find $p^*_1=p\up{\pi}_1(\theta_0)$ (Proposition~\ref{pro:pi}) as well as some consecutive optimal control actions (Corollary~\ref{cor:pi}) without needing to explicitly solve Problem~\eqref{mso}. The algorithm requires a preset search accuracy $\epsilon$ and is described as follows:
\begin{enumerate}
    \item Initialize a search range $L$ and $R$ with which we are confident that $L \leq \theta_0 \leq R$; 
    \item Set $x$ to $(L+R)/2$. If $R-L < \epsilon$, return $x$ as the optimal Lagrangian dual value and $p\up{\pi}_t(x)$ as the optimal storage control up to time step $t$;
    \item Run the following iterative simulation
    \begin{enumerate}
        \item Set $1\to t$ and $e_0\to \sigma_0(x)$;
        \item Calculate $p\up{\pi}_t(x)$ using Eq.~\eqref{es:policy};
        \item Calculate $\sigma_t(x)$ using Eq.~\eqref{eq:th_soc};
        \item If $\sigma_t(x) \geq E$, set $x \to R$, go to Step 2);
        \item If $\sigma_t(x) \leq 0$, set $x \to L$, go to Step 2);
        \item If $t<T$, set $t+1 \to t$ and go to Step b);
        \item If $x \geq -c_T(\sigma_T(x))$, set $x \to R$;
        \item If $x < -c_T(\sigma_T(x))$, set $x \to L$;
    \end{enumerate}
    \item Go to Step 2).
\end{enumerate}
An example of a confident search range is that we can assume stored energy always has a positive value and choose $L = 0$ and $R = \max\{o_t(p)/\eta|t\in\mathcal{T}, p\in [-P, P]\}$.

This algorithm achieves the following complexity results:
\begin{enumerate}
    \item \emph{Constant space complexity:} The algorithm achieves $O(1)$ space complexity with respect to the search range and the look-ahead duration $T$, because the equality relationship between $x$ and $\theta_0$ can be identified using only the current simulation result $\sigma_t(x)$ so that previous simulation results are not required to be stored.
    \item \emph{Worst-case linear run-time complexity:} The algorithm achieves a worst-case $O(n)$ complexity with respect to the look-ahead horizon $T$ since the worst-case scenario is to simulate all operations steps from $t=1$ to $T$ during each search, but may terminate before reaching $T$ as stated in step 3-d and 3-e. It also achieves $O(\log n)$ time complexity with respect to the search range for using a binary search algorithm.
\end{enumerate}


\subsection{Non-simultaneous Charging and Discharging}

By far we have concluded the optimal storage control when relaxing the non-simultaneous charging and discharging constraint~\eqref{mso:c1}. A sufficient condition for relaxing this constraint without sacrificing result optimality is illustrated in the following proposition:
\begin{proposition}
     A sufficient condition for simultaneous charging and discharging to happen is the Lagrangian dual being negative, i.e., if $\theta_{t-1} \geq 0$ then $p^+_tp^-_t = 0$.
\end{proposition}
\begin{proof}
From \eqref{es:policy}, it is trivial to see that for simultaneous charging and discharge to happen, both terms in \eqref{es:policy} must take non-zero values, hence
\begin{gather}
    \varphi_t(-x\eta) < 0 < \varphi_t(-x/\eta)
\end{gather}
and since $\varphi_t$ is a monotonic increasing function and $0 < \eta \leq 1$, it follows
\begin{align}
    -x\eta < -x/\eta
\end{align}
hence $x$ must be less than zero.
\end{proof}
Recall that $\theta_t$ is the value of the stored energy, hence $\theta_t$ being negative indicates the stored energy has a negative value, i.e., we have an intention to store as less energy as possible. Thus, when \eqref{mso:c1}, the storage can charging and discharging at the same time and use round-trip efficiency loss to consume excessive energy even when the storage is full and has no more storage space. This intuition is useful when deciding whether simultaneous charging and discharging should be considered when formulating the problem, for example, this constraint should be considered in price arbitrage for markets with frequent negative prices.

A common method for enforcing non-simultaneous charging and discharging is to add auxiliary binary variables $\bm{v}=\{v_t\}$ such that the storage can only charge or discharge at one time, as
\begin{align}\label{mso:c0b}
    p^+_t <= Pv_t,\; p^-_t <= P(1-v_t),\; v_t \in \{0,1\}
\end{align}
making \eqref{mso} a mixed-integer programming problem, the Lagrangian dual can thus be calculated given a set of fixed $\bm{v}$.

We assume set $\mathcal{V}$ as the set of all reasonable charging status results, in which $v_t$ may be either 0 or 1 during all periods when simultaneous charging and discharge occur, i.e., when \eqref{mso:c0} must be enforced. It is worth noting that the optimal result for $v\bm{v}$ must be in $\mathcal{V}$. We denote $\theta_0(\bm{v})$ as the resulting Lagrangian dual associated with the charging status $\bm{v}$, then the following proposition stands:
\begin{proposition}
\label{pro:dual_b}
     For all $\bm{v}\in\mathcal{V}$, $\underline{\theta}_0 \leq \theta_0(\bm{v})\leq \overline{\theta}_0$, where $\underline{\theta}_0$ is calculated using Algorithm 1 by replacing \eqref{es:policy} with
     \begin{subequations}
     \begin{align}
             p^{\pi-}_t(x) = p^+_t(x)\mathbf{1}_{[p^-_t(x) == 0]} - p^-_t(x)\label{es:policy2}
     \end{align}
     and $\underline{\theta}_0$ is calculated using Algorithm 1 by replacing \eqref{es:policy} with
     \begin{align}
             p^{\pi+}_t(x) = p^+_t(x) - p^-_t(x)\mathbf{1}_{[p^+_t(x) == 0]}\label{es:policy3}
     \end{align}
     \end{subequations}
     where $\mathbf{1}_{[x]} = \{\text{ 0 if $x$ is true and 1 otherwise} \}$ is the indicator function.
\end{proposition}
\begin{proof}
First note that compared to \eqref{es:policy}, \eqref{es:policy2} enforces the storage to charge whenever charging and discharging components are both non-zero. Thus when using \eqref{es:policy2} to simulate the battery operation in Algorithm 1, the resulting SoC must always be lower than using any $\bm{v}\in\mathcal{V}$. Thus according to Theorem~\ref{th1}, the resulting  dual $\theta_0$ must be no smaller than any dual $ \theta_0(\bm{v})$ using charging status $\bm{v}\in\mathcal{V}$. Vice versa, when using \eqref{es:policy3}, the battery prefer discharging over charging, resulting in a lower bound for $\theta_0$.  Note that in \eqref{es:policy2} and \eqref{es:policy3}, $p^{\pi}_t(x)$ is still a monotonic decreasing function to $x$ since $p^+_t(x) \geq 0$ and $p^-_t(x) \geq 0$ for all $x$, $p^+_t(x)$ is an decreasing function and $p^-_t(x)$ is an increasing function. Hence the SoC series is still monotonic increasing with respect to $x$, and the convergence optimality of Algorithm 1 will not be effected.
\end{proof}
And the result on the dual binding can be extended to bind the primal control results:
\begin{proposition}
\label{pro:prim_b}
   Let $\underline{p}_t = p^{\pi-}_1(\underline{\theta}_{0})$  and $\overline{p}_t = p^{\pi+}_t(\overline{\theta}_{0})$, and $p^*_t$ is the optimal control, then $\underline{p}_t \leq p^*_t \leq \overline{p}_t$ for all $t\in[1,\;\tau]$.
\end{proposition}
\begin{proof}
First given the dual $\theta_0$, the optimal control $p^*_1$ is either $p^+_t(\theta_0)$ or $-p^-_t(\theta_0)$ as in \eqref{es:policy}, since the battery either charges or discharges. Then according to Proposition~\ref{pro:dual_b}, we know the dual must be within $[\underline{\theta}, \overline{\theta}]$. Then we have
\begin{align}
    \underline{p}_t \leq -p^-_t(\theta_0) \leq p^+_t(\theta_0) \leq \overline{p}_t
\end{align}
which proved this Proposition.
\end{proof}

Hence Proposition~\ref{pro:dual_b} and Proposition~\ref{pro:prim_b} provide bounds for the primal and dual result of the non-convex battery control problem using the proposed algorithm that has worst-case linear time complexity.

\section{Numerical Simulation}

We use randomly generated data sets to compare the proposed algorithm to solving Problem~\ref{mso} with different objectives using Gurobi~\cite{gurobi} (model generated using CVX~\cite{cvx}). All simulations are performed in Matlab~\cite{matlab} on a 2.3 GHz machine with 16GB memory.The storage parameter is set as $P = 1$~p.u., $E = 4$~p.u., $e_0 = 2$~p.u., $\eta = 0.92$, and the terminal cost function is set to $C_T(e_T)= (E-e_T)^2/2$. The accuracy of the search algorithm is set to $\epsilon = 10^{-3}$.

\subsection{Piece-wise linear objectives}

\begin{table}[!t]
    \centering
    \caption{Piecewise-linear simulation results.}
    \begin{tabular}{l r r r r r r}
        \hline
        \hline\Tstrut
          & \multicolumn{3}{c}{ CVX+Gurobi }  & \multicolumn{3}{c}{ Proposed} \\
          Trials & $\theta_0$ & $p^*_1$ & cpu [ms] & $\theta_0$ & $p^*_1$ & cpu [ms]
        \\
        \hline\Bstrut\Tstrut
        & \multicolumn{6}{c}{ $T=10$, $J=100$ } \\
        \hline\Bstrut\Tstrut
        1 & 17.53 & 0.6 & 408.8 & 17.53 & 0.6 & 0.1 \\
        2 & 16.82 & 0.0 & 278.1 & 16.82 & 0.0 & 0.1  \\
        3 & 14.59 & 0.3 & 274.5 & 14.59 & 0.3 & 0.1   \\
        4 & 14.21 & 0.1 & 281.7 & 14.21 & 0.1 & 0.1  \\
        5 & 7.27 & -0.5 & 283.7 & 7.27 & -0.5 & 0.1  \\
                \hline\Bstrut\Tstrut
        & \multicolumn{6}{c}{ $T=10$, $J=1000$ } \\
        \hline\Bstrut\Tstrut
        6 & 18.67 & 0.0 & 1851.5 & 18.66 & 0.0 & 0.1 \\
        7 & 20.28 & 0.8 & 1791.5 & 20.28 & 0.8 & 0.2 \\
        8 & 19.50 & 0.0 & 1788.2 & 19.50 & 0.0 & 0.1   \\
        9 & 17.44 & -0.7 & 1868.4 & 17.44 & -0.7 & 0.2  \\
        10 & 19.00 & 0.0 & 1868.7 & 19.00 & 0.0 & 0.2  \\
                \hline\Bstrut\Tstrut
        & \multicolumn{6}{c}{ $T=100$, $J=1000$ } \\
        \hline\Bstrut\Tstrut
        11 & 19.09 & 0.6 & 18497.7 & 19.09 & 0.6 & 0.1 \\
        12 & 19.95 & 0.8 & 19108.2 & 19.95 & 0.8 & 0.2  \\
        13 & 19.44 & 0.0 & 18786.1 & 19.44 & 0.0 & 0.1   \\
        14 & 18.91 & -1.0 & 19263.2 & 18.91 & -1.0 & 0.2  \\
        15 & 19.55 & 0.0 & 19080.4 & 19.55 & 0.0 & 0.2 \\
         \hline
    \end{tabular}
    \label{tab:sim4}
\end{table}

\begin{figure}[!htb]
    \centering
    \includegraphics[trim={10mm 0 10mm 3mm},clip, width = .95\columnwidth]{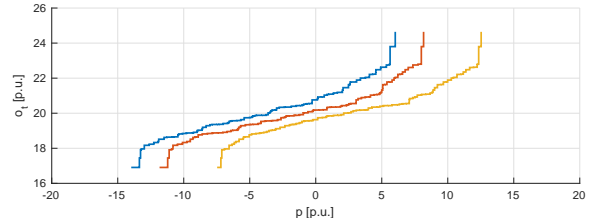}
    \caption{Three examples of generated piecewise linear cost curves plotted as derivative of the objective function $o_t$.}
    \label{fig:pwl_example}
\end{figure}

In this section the proposed algorithm is compared with Gurobi using piecewise linear objective function inspired by the supply curves in power system economic dispatches~\cite{kirschen2018fundamentals}. The derivative of the objective function $o_t$ is written as
\begin{align}
    o_t(p) = c_j \text{ if } q_{j-1} \leq p < q_j
\end{align}
where $j\in [1,J]$ is the piecewise segment index, $J$ is the number of segments, $c_j$ is the marginal cost (derivative) of the system when $p$ is between quantities $q_{j-1}$ and $q_j$, and the objective is convex if $c_i \leq c_j$ and $q_i \leq q_j$ for all $i<j$, $i,j\in [1,J]$. Some examples of the generated cost curve are plotted in Fig.~\ref{fig:pwl_example}.
\begin{figure}[!t]
    \centering
    \subfloat[]{
        \includegraphics[trim={10mm 0 10mm 3mm},clip, width = .95\columnwidth]{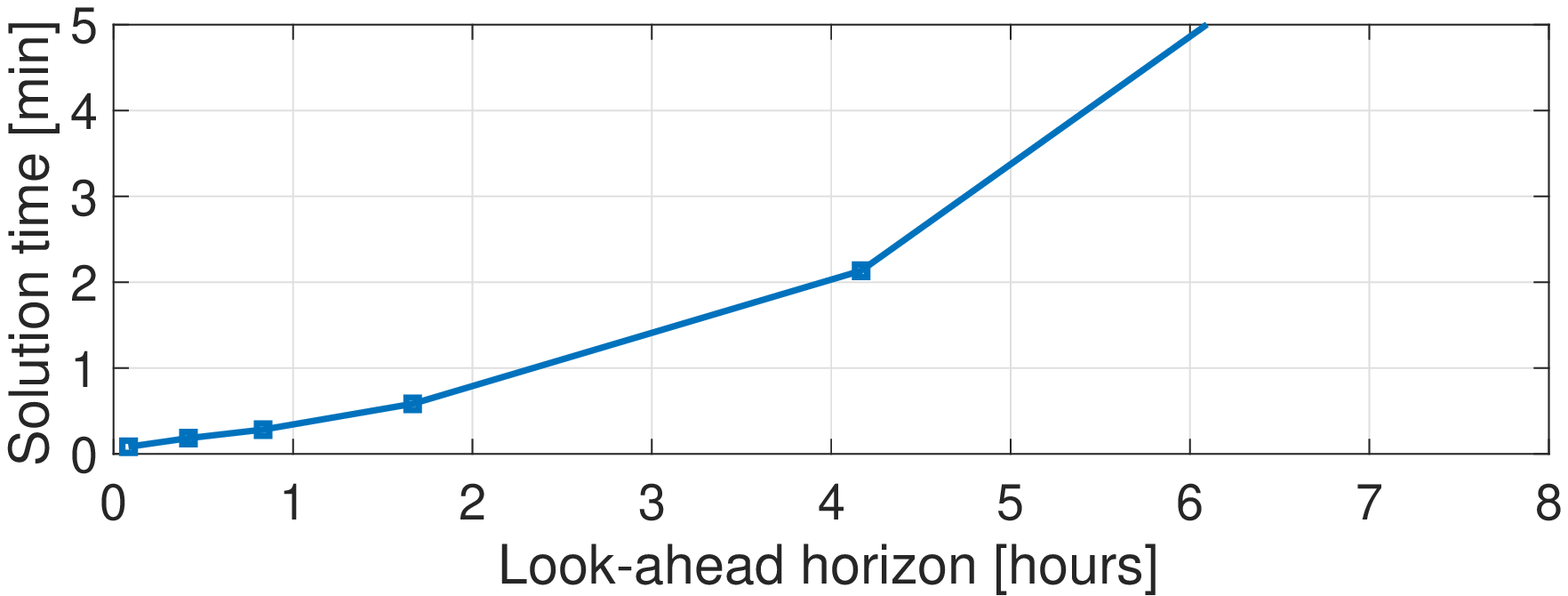}
        \label{fig:comp1}
    }
    \\
    \subfloat[]{
        \includegraphics[trim={10mm 0 10mm 3mm},clip, width = .95\columnwidth]{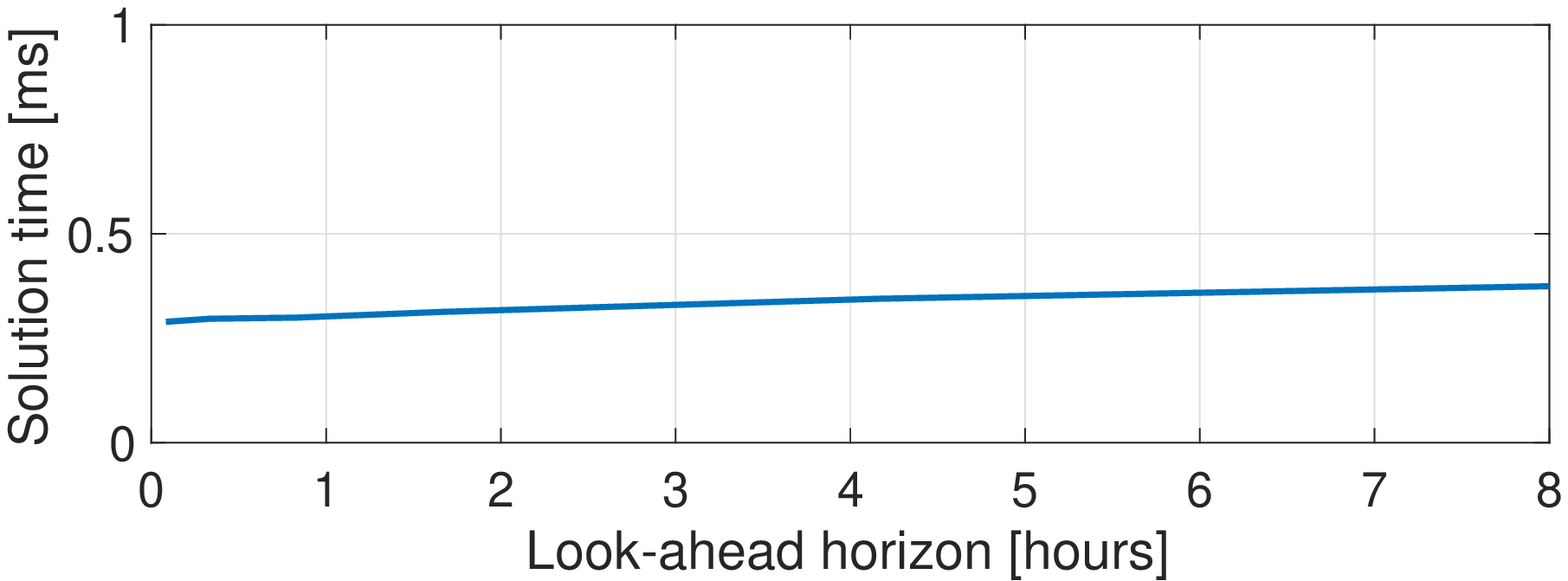}
        \label{fig:comp2}
    }
    \caption{\footnotesize Average computation time for solving look-ahead economic dispatch with 5000 cost segments per five minutes using (a) Gurobi and (b) the proposed algorithm. Note that (a) is plotted in minutes while (b) is plotted in milliseconds. In (a), when the look-ahead horizon is beyond 6 hours, the problem takes more than five minutes to solve, which is infeasible for real-time economic dispatch that must be finished within five minutes.}
    \label{fig:comp}
\end{figure}
Similar to the quadratic results, we test the proposed algorithm and Gurobi using different settings and the results are demonstrated in Table~\ref{tab:sim4}, where trials 1--5 have 10 time steps and 100 cost segments $T=10$, $J=100$, trials 6--10 have 10 time steps and 1,000 cost segments $T=10$, $J=1000$, trials 11--15 have 100 time steps and 1,000 cost segments $T=100$, $J=1000$. The result shows the proposed algorithm obtains the same results in all trials compared to Gurobi, while being hundreds or even thousands of times faster.
In particular, in trials 11-15 Gurobi needs around 18 seconds to complete the computation, while the proposed algorithm finishes below 1ms.

In Fig.~\ref{fig:comp}, we further test the computation speed of both methods in solving  look-ahead economic dispatches over the size of realistic power systems with 5000 cost segments per five minute dispatch interval. The result shows that the computation time of Gurobi increases significantly with respect to the look-ahead horizon, and in particular at the 6 hour look-ahead, the problem takes more than 5 minutes to solve which is not feasible since the economic dispatch must be calculated within 5 minutes. In contrast, with our proposed algorithm, the average solution speed is below 0.5 milliseconds for look-ahead horizons up to 8 hours, providing a speed-up up to 100,000 times.

\subsection{Negative Lagrangian dual example}

\begin{table}[!t]
    \centering
    \caption{Negative Lagrangian dual results.}
    \begin{tabular}{l r r r r r r}
        \hline
        \hline\Tstrut
          & \multicolumn{3}{c}{ Dual results }  & \multicolumn{3}{c}{ Control results } \\
          Trials & $\underline{\theta}_0$ & ${\theta}^*_0$ & $\overline{\theta}_0$  & $\underline{p}_1$ & $p^*_1$ & $\overline{p}_1$
        \\
        \hline\Bstrut\Tstrut
        1 & -10.7 & -10.6 & -10.2 & -0.14 & 0.10 & 0.11\\
        2 & -12.9 & -12.5 & -11.8 & 0.49 & 0.53 & 0.56\\
        2 & -8.5 & -8.2 & -8.2 & -0.17 & -0.17 & -0.16\\
        4 & -4.4 & -4.4 & -4.0 & -1.00 & -1.00 & -1.00 \\
        5 & -15.0 & -14.6 & -14.5 & 0.89 & 0.91 & 0.97 \\
        \hline
    \end{tabular}
    \label{tab:sim1}
\end{table}

We consider the following quadratic objective function
\begin{align}
    O_t = \frac{\alpha_t}{2} (\beta_t - p)^2
\end{align}
where $\alpha_t$ are randomly generated between $[0, 10]$, and $\beta_t$ between $[-10,0]$. Recall that negative sign is for charging the battery, hence this is a generation tracking problem where the storage wishes to absorb as much energy as possible. This will result in negative dual prices for the storage and simultaneous charging and discharging will occur if not constrained. Table~\ref{tab:sim1} shows the simulation result for five trails including upper and lower bounds on the dual and control, where ${\theta}^*_0$ and $p^*_1$ is calculated using Gurobi with mixed-integer quadratic programming under default settings, solving \eqref{mso} using the integer constraint \eqref{mso:c0b} for enforcing constraint \eqref{mso:c0}. In all test trails, the primal and dual result fall within the calculated range. In terms of computation speed, the proposed method all solves less than 1 millisecond, while the benchmark method using Gurobi may need up to several minutes to solve the problem depending on the problem size (thousands of steps), due to solving a mixed-integer quadratic programming problem.


\section{Conclusion}
This paper proposed a novel algorithm for solving look-ahead control for energy storage. The numerical results illustrate that the algorithm provides computation speed in milliseconds for controlling a single energy storage device over an extended planning period. In future research, we plan on expanding this method to controlling multiple energy storage devices subject to network constraints. Moreover, using the generalized terminal state function we plan on incorporating this algorithm into scenario-based stochastic programming or dynamic programming. In addition, our results connects the optimal control with the Lagrangian multiplier associated with the state-of-charge constraint, which we will further explore to provide key insights into designing future electricity pricing and distributed control schemes.

\bibliographystyle{IEEEtran}	
\bibliography{IEEEabrv,literature}		

\begin{thebibliography}{10}
\providecommand{\url}[1]{#1}
\csname url@samestyle\endcsname
\providecommand{\newblock}{\relax}
\providecommand{\bibinfo}[2]{#2}
\providecommand{\BIBentrySTDinterwordspacing}{\spaceskip=0pt\relax}
\providecommand{\BIBentryALTinterwordstretchfactor}{4}
\providecommand{\BIBentryALTinterwordspacing}{\spaceskip=\fontdimen2\font plus
\BIBentryALTinterwordstretchfactor\fontdimen3\font minus
  \fontdimen4\font\relax}
\providecommand{\BIBforeignlanguage}[2]{{%
\expandafter\ifx\csname l@#1\endcsname\relax
\typeout{** WARNING: IEEEtran.bst: No hyphenation pattern has been}%
\typeout{** loaded for the language `#1'. Using the pattern for}%
\typeout{** the default language instead.}%
\else
\language=\csname l@#1\endcsname
\fi
#2}}
\providecommand{\BIBdecl}{\relax}
\BIBdecl

\bibitem{korpas2006operation}
M.~Korpas and A.~T. Holen, ``Operation planning of hydrogen storage connected
  to wind power operating in a power market,'' \emph{IEEE Transactions on
  Energy Conversion}, vol.~21, no.~3, pp. 742--749, 2006.

\bibitem{li2016connecting}
N.~Li, C.~Zhao, and L.~Chen, ``Connecting automatic generation control and
  economic dispatch from an optimization view,'' \emph{IEEE Transactions on
  Control of Network Systems}, vol.~3, no.~3, pp. 254--264, 2016.

\bibitem{zhang2015optimal}
B.~Zhang, A.~Y. Lam, A.~D. Dom{\'\i}nguez-Garc{\'\i}a, and D.~Tse, ``An optimal
  and distributed method for voltage regulation in power distribution
  systems,'' \emph{IEEE Transactions on Power Systems}, vol.~30, no.~4, pp.
  1714--1726, 2015.

\bibitem{khalid2010model}
M.~Khalid and A.~Savkin, ``A model predictive control approach to the problem
  of wind power smoothing with controlled battery storage,'' \emph{Renewable
  Energy}, vol.~35, no.~7, pp. 1520--1526, 2010.

\bibitem{krishnamurthy2018energy}
D.~Krishnamurthy, C.~Uckun, Z.~Zhou, P.~R. Thimmapuram, and A.~Botterud,
  ``Energy storage arbitrage under day-ahead and real-time price uncertainty,''
  \emph{IEEE Transactions on Power Systems}, vol.~33, no.~1, pp. 84--93, 2018.

\bibitem{shi2018using}
Y.~Shi, B.~Xu, D.~Wang, and B.~Zhang, ``Using battery storage for peak shaving
  and frequency regulation: Joint optimization for superlinear gains,''
  \emph{IEEE Transactions on Power Systems}, vol.~33, no.~3, pp. 2882--2894,
  2018.

\bibitem{xu2016hierarchical}
Z.~Xu, W.~Su, Z.~Hu, Y.~Song, and H.~Zhang, ``A hierarchical framework for
  coordinated charging of plug-in electric vehicles in china,'' \emph{IEEE
  Transactions on Smart Grid}, vol.~7, no.~1, pp. 428--438, 2016.

\bibitem{megel2014scheduling}
O.~M{\'e}gel, J.~L. Mathieu, and G.~Andersson, ``Scheduling distributed energy
  storage units to provide multiple services,'' in \emph{Power Systems
  Computation Conference (PSCC), 2014}.\hskip 1em plus 0.5em minus 0.4em\relax
  IEEE, 2014, pp. 1--7.

\bibitem{epri_2016}
\BIBentryALTinterwordspacing
{EPRI}, ``Wholesale electricity market design initiatives in the united states:
  Survey and research needs,'' 2016. [Online]. Available:
  \url{https://www.epri.com/#/pages/product/3002009273/}
\BIBentrySTDinterwordspacing

\bibitem{zhao2018multi}
J.~Zhao, T.~Zheng, and E.~Litvinov, ``A multi-period market design for markets
  with intertemporal constraints,'' \emph{arXiv preprint arXiv:1812.07034},
  2018.

\bibitem{castillo2013profit}
A.~Castillo and D.~F. Gayme, ``Profit maximizing storage allocation in power
  grids,'' in \emph{52nd IEEE Conference on Decision and Control}.\hskip 1em
  plus 0.5em minus 0.4em\relax IEEE, 2013, pp. 429--435.

\bibitem{wang2011wind}
J.~Wang, A.~Botterud, R.~Bessa, H.~Keko, L.~Carvalho, D.~Issicaba, J.~Sumaili,
  and V.~Miranda, ``Wind power forecasting uncertainty and unit commitment,''
  \emph{Applied Energy}, vol.~88, no.~11, pp. 4014--4023, 2011.

\bibitem{papavasiliou2018application}
A.~Papavasiliou, Y.~Mou, L.~Cambier, and D.~Scieur, ``Application of stochastic
  dual dynamic programming to the real-time dispatch of storage under renewable
  supply uncertainty,'' \emph{IEEE Transactions on Sustainable Energy}, vol.~9,
  no.~2, pp. 547--558, 2018.

\bibitem{gao2002dynamic}
L.~Gao, S.~Liu, and R.~A. Dougal, ``Dynamic lithium-ion battery model for
  system simulation,'' \emph{IEEE transactions on components and packaging
  technologies}, vol.~25, no.~3, pp. 495--505, 2002.

\bibitem{cruise2014optimal}
J.~Cruise, L.~Flatley, R.~Gibbens, and S.~Zachary, ``Optimal control of storage
  incorporating market impact and with energy applications,'' \emph{arXiv
  preprint arXiv:1406.3653}, 2014.

\bibitem{hashmi2017optimal}
M.~U. Hashmi, A.~Mukhopadhyay, A.~Bu{\v{s}}i{\'c}, and J.~Elias, ``Optimal
  control of storage under time varying electricity prices,'' in \emph{2017
  IEEE International Conference on Smart Grid Communications
  (SmartGridComm)}.\hskip 1em plus 0.5em minus 0.4em\relax IEEE, 2017, pp.
  134--140.

\bibitem{bertsekas2005dynamic}
D.~P. Bertsekas, \emph{Dynamic programming and optimal control}.\hskip 1em plus
  0.5em minus 0.4em\relax Athena scientific Belmont, MA, 2005, vol.~1, no.~3.

\bibitem{boyd2004convex}
S.~Boyd and L.~Vandenberghe, \emph{Convex optimization}.\hskip 1em plus 0.5em
  minus 0.4em\relax Cambridge university press, 2004.

\bibitem{gurobi}
\BIBentryALTinterwordspacing
L.~Gurobi~Optimization, ``Gurobi optimizer reference manual,'' 2018. [Online].
  Available: \url{http://www.gurobi.com}
\BIBentrySTDinterwordspacing

\bibitem{cvx}
M.~Grant and S.~Boyd, ``{CVX}: Matlab software for disciplined convex
  programming, version 2.1,'' \url{http://cvxr.com/cvx}, Mar. 2014.

\bibitem{matlab}
\BIBentryALTinterwordspacing
Mathworks, ``Matlab r2018a,'' 2018. [Online]. Available:
  \url{http://www.mathworks.com}
\BIBentrySTDinterwordspacing

\bibitem{kirschen2018fundamentals}
D.~S. Kirschen and G.~Strbac, \emph{Fundamentals of power system
  economics}.\hskip 1em plus 0.5em minus 0.4em\relax John Wiley \& Sons, 2018.

\end{thebibliography}

\appendix
\section{Proof of Theorem 1}
We start by showing that when moving to the next control step, the value of the Lagrangian $\theta_t$ only changes after $e^*_t$ is reaching the upper or lower SoC bound, or more specifically:
    \begin{subequations}\label{eq:theta}
    \begin{align}
        \theta_{t-1} = \theta_t\quad &\text{if $\overline{\nu}_{\tau} = \underline{\nu}_{\tau} = 0$}\\
        \theta_{t-1} < \theta_t\quad &\text{if $\overline{\nu}_{\tau} > 0$}\\
        \theta_{t-1} > \theta_t\quad &\text{if $\underline{\nu}_{\tau} > 0$}
    \end{align}
    \end{subequations}
Hence, it is trivial to see that if $0 \leq \sigma_{\tau}(\theta_0) \leq E$ $\forall$ $\tau\in[1,t]$ then $\theta_{t-1} = \theta_0$, leading $p\up{\pi}_t(\theta_{t-1}) = p\up{\pi}_t(\theta_{0}) = p^*_t$ according to Proposition~\ref{pro:pi}.

We will do the proof separately for three possible cases of $e^*_t$: 1) $e^*_t$ never reaches upper or lower bound with all $\underline{\nu}_t$ and $\overline{\nu}_t$ equal to zero; 2) $e^*_t$ reached upper bound first; 3) $e^*_t$ reached lower bound first. These three cases are illustrated in Fig.~\ref{fig:th1}.

\subsubsection{$0<e^*_t < E$ $\forall$ $t\in [1,T]$}
This cover the cases when $e^*_t$ never reached the upper or lower bound, which from \eqref{eq:theta} we know $\theta_0 = \theta_1 = \dotsc = \theta_T$, hence  $p^*_t = p\up{\pi}_t(\theta_0)$ and $e^*_t = \sigma_t(\theta_0)$ for all $t\in[1,T]$, and in particular for $t=T$ we have
\begin{align}\label{eq:th1_T}
\theta_0 = \theta_T = -c_T(e^*_T) = -c_T(\sigma_T(\theta_0))
\end{align}
according to \eqref{mfv_eq4} and the aforementioned result.
Since $O_t$ and $c_T$ are convex, $o_t$, $c_T$, and $\varphi_t$ (inverse of $o_t$) are monotonic increasing functions, it follows
\begin{subequations}\label{eq:th1a}
\begin{align}
    x &\geq \theta_0 \\
    p\up{\pi}_t(x) &\leq p\up{\pi}_t(\theta_0) : \text{convexity, see \eqref{pro:pi}} \\
    \sigma_t(x) &\geq \sigma_t(\theta_0) : \text{see \eqref{eq:th_soc}} \\
    c_T(\sigma_t(x)) &\geq c_T(\sigma_t(\theta_0)) : \text{$c_T(\cdot)$ both sides } \\
    c_T(\sigma_T(x)) &\geq -\theta_0 : \eqref{eq:th1_T} \\
    c_T(\sigma_T(x)) &\geq -x :  -\theta_0 \geq -x\\
    x &\geq -c_T(\sigma_T(x))
\end{align}
\end{subequations}
meaning if $x \geq -c_T(\sigma_T(x))$ then $x \geq \theta_0 $, thus we proved condition 1-b in the Theorem. Similarly we can prove condition 2-b starting with $x\leq \theta_0$. Also it is trivial to see that if any $\sigma_t(x)$ goes above $E$ then in this case we know $\sigma_t(x) > \sigma_t(\theta_0)$ hence $x > \theta_0$, and vice versa for $\sigma_t(x)$ goes below 0, hence we proved condition 1-a and 2-a. It is also trivial to see that if $x = c_T(\sigma_T(x))$, then the KKT condition is satisfied and $x=\theta_0$, which proves condition 3. Thus all conditions in this theorem are proved for this this case.

\subsubsection{$\exists \tau\in[1,T]$ s.t.  $\overline{\nu}_{\tau} > 0$ and $0<e^*_t < E$ $\forall$ $t\in [1,\tau)$} This covers the cases when $e^*_t$ reached the upper bound first.
An example of this case in shown in Fig.~\ref{fig:th1_example}.
\begin{figure}[!htb]
    \centering
    \includegraphics[trim={15mm 0mm 10mm 5mm},clip, width = .95\columnwidth]{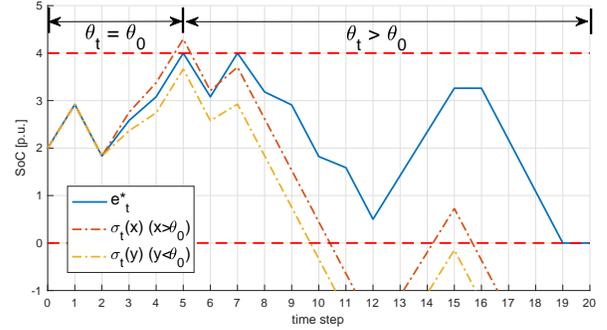}
    \caption{An explanatory example when $e^*_t$ reached upper bound first, the upper and lower SoC bounds are plotted wit red dash. Besides $e^*_t$ plotted with the blue line, two comparative SoC series $\sigma_t(x)$ and $\sigma_t(y)$ are included in this figure with $y<\theta_0<x$. It is clearly from the figure that $\sigma_t(x)$ reached upper bound first indicating $x \geq \theta_0$, while $\sigma_t(y)$ reached lower bound first such that $y \leq \theta_0$, as stated in Theorem~\ref{th1}.}
    \label{fig:th1_example}
\end{figure}
Now from \eqref{eq:th1a} we can conclude if $\exists \tau\in[1,T]$ s.t.  $\overline{\nu}_{\tau} > 0$ and $0<e^*_t < E$ $\forall$ $t\in [1,\tau)$, then the same condition must be satisfied for all $x\geq \theta_0$, hence condition 1-a is proved. And from \eqref{eq:theta} we know after reaching the upper bound, all the following $\theta$ values will be greater than $\theta_0$ until $e^*_t$ reaches the lower bound or till the end of the operation $T$ (i.e., $e^*_t$ never reaches the lower bound). Without loss of generality, let $\gamma$ be the time that $e^*_t$ first reaches the lower bound or the end of the operation period, i.e., $e^*_t > 0$ $\forall$ $t\in [1,\gamma)$, it follows
\begin{subequations}
\begin{align}
    x &\leq \theta_0 \\
    x &\leq \theta_t\; \forall\; t\in [1,\gamma)\\
    p\up{\pi}_t(x) &\geq p\up{\pi}_t(\theta_{t-1}) : \text{see \eqref{pro:pi}} \\
    p\up{\pi}_t(x) &\geq p^*_t : \text{Proposition~\ref{pro:pi}} \\
    \sigma_t(x) &\leq e^*_t
\end{align}
\end{subequations}
which means $\sigma_t(x)$ either will go below $0$ (condition 2-a) or $\sigma_T(x)\leq e^*_T$ which leads to $x \leq c_T(\sigma_T(x))$ (condition 2-b) according to \eqref{eq:th1a}, hence we proved this theorem for this case.

\subsubsection{$\exists \tau\in[1,T]$ s.t.  $\underline{\nu}_{\tau} > 0$ and $0<e^*_t < E$ $\forall$ $t\in [1,\tau)$} This covers the case when $e^*_t$ reaches the lower bound first. 
This is a mirror proof to the previous case while inverting the upper and lower bound logic, hence this proof is omitted.

\end{document}